\documentclass[12pt,reqno]{amsart}

\usepackage{hyperref}
\usepackage{color}

\newtheorem{Theorem}{Theorem}

\newtheorem{Lemma}[Theorem]{Lemma}

\newtheorem{Conjecture}[Theorem]{Conjecture}
\theoremstyle{remark}
\newtheorem*{Remark}{Remark}

\newtheorem*{rems}{Remarks} 

\allowdisplaybreaks 

\newcommand{\la}{\lambda}

\author[G.~Bhatnagar]{Gaurav Bhatnagar}
\address{Fakult\"at f\"ur Mathematik\\
Universit\"at Wien\\
Oskar-Morgenstern-Platz 1\\
A-1090 Vienna, Austria.}
\curraddr{School of Physical Sciences,
Jawaharlal Nehru University,
Delhi, India.}
\email{bhatnagarg@gmail.com}

\author[M.\ J.\ Schlosser]{Michael J.\ Schlosser}
\address{Fakult\"at f\"ur Mathematik\\
Universit\"at Wien\\
Oskar-Morgenstern-Platz~1\\
A-1090 Vienna, Austria}
\email{michael.schlosser@univie.ac.at}

\title[A partial theta function Borwein conjecture]{A partial theta function
 Borwein conjecture}
 


\dedicatory{Dedicated to George Andrews on the occasion
of his 80th birthday}
\begin{document}

\begin{abstract}
We present an infinite family of Borwein type $+ - - $ conjectures.
The expressions in the conjecture are related to multiple
basic hypergeometric series with Macdonald polynomial argument.
\end{abstract}

\maketitle



\section{Introduction}
The so-called Borwein conjectures, due to Peter Borwein (circa 1990), were  popularized by Andrews \cite{An1995}. The first of these concerns the expansion of finite products of the form
$$(1-q)(1-q^2)(1-q^4)(1-q^5)(1-q^7)(1-q^8)\cdots$$
into a power series in $q$ and the sign pattern displayed by the coefficients. 
In June 2018, in a conference at Penn State celebrating Andrews' 80th birthday, 
Chen Wang, a young Ph.D. student studying at the University of Vienna, announced that he has vanquished the first of the Borwein conjectures. In this paper, we 
propose another set of Borwein-type conjectures.
The conjectures here are consistent with the first two Borwein conjectures, and one given by Ismail, Kim and Stanton \cite{IKS1999, DS1999}.
At the same time, they do not appear to be very far from these conjectures in form and content. However, they are on different lines from other extensions of Borwein conjectures considered in 
\cite{BW2005, DB1996, IKS1999, RP2009, DS1999,  SOW2001b, SOW2003b}.

Borwein's first conjecture may be stated as follows:
the polynomials $A_n(q)$, $B_n(q)$, and $C_n(q)$ defined by
\begin{equation}\label{borwein1}
\prod_{i=0}^{n-1} (1-q^{3i+1}) (1-q^{3i+2}) = 
A_n(q^3)-qB_n(q^3)-q^2C_n(q^3),
\end{equation}
each have non-negative coefficients.  This is the one now settled by Wang~\cite{Wang2019}. We say that the polynomial on the left-hand side
satisfies the Borwein $+ - - $ condition.  

Our first conjecture considers products of the form
$$
\prod_{i=0}^{n-1} (1-q^{3i+1}) (1-q^{3i+2})
\prod_{j=1}^m \prod_{i=-n}^{n-1} (1-p^jq^{3i+1})(1- p^jq^{3i+2}) 
.
$$
Computational evidence suggests that for fixed $k$, the coefficient of $p^k$ (a Laurent polynomial in $q$) satisfies the Borwein $+ - -$ condition for $n$ large enough. For $m=0$, this reduces to the left-hand side of \eqref{borwein1}.

This paper is organized as follows. In Section~2 we present a precise statement of this conjecture and outline the computational evidence for this conjecture. We also make another---even more general---conjecture, which is motivated by the first two Borwein conjectures, and Andrews' refinement of these conjectures. Our third and most general conjecture is motivated by Ismail, Kim and Stanton \cite[Conjecture~1]{IKS1999} (see also Stanton~\cite[Conjecture~3]{DS1999}).
In Section~3, we make some remarks concerning the connection
to multiple basic hypergeometric series with Macdonald polynomial argument.

\section{The conjectures}

Let $a$, $p$ and $q$ be formal variables. We shall work in the ring of Laurent polynomials in $q$.
For $n$ being a non-negative integer or infinity,
the $q$-shifted factorial is defined as follows:
\begin{equation*}
(a;q)_n:=\prod_{j=0}^{n-1}(1-a q^j).
\end{equation*}
For convenience, we write
\begin{equation*}
(a_1,\dots,a_m;q)_n:=\prod_{k=1}^m(a_k;q)_n
\end{equation*}
for products of $q$-shifted factorials.
With this notation, our first conjecture can be stated as follows. 
\begin{Conjecture}\label{conj1}
Let $m$ and $k$ be non-negative integers.
Let the Laurent polynomials $A_{m,n,k}(q)$, $B_{m,n,k}(q)$, and
$C_{m,n,k}(q)$ be defined by
\begin{align}\label{eq:conj1}
&(q, q^2;q^3)_n \prod_{j=1}^m(p^jq,p^jq^2;q^3)_n (p^jq^{-1},p^jq^{-2};q^{-3})_n
\notag\\
&=\sum_{k\geq 0} p^k\left[A_{m,n,k}(q^3)-q B_{m,n,k}(q^3)-q^2
C_{m,n,k}(q^3)\right].
\end{align}
Then for each $m, k\geq 0$,  there is a non-negative integer $N_{m,k}$, such that
if $n\geq N_{m, k}$ then the Laurent polynomials 
$A_{m,n,k}(q)$, $B_{m,n,k}(q)$, and
$C_{m,n,k}(q)$ have non-negative coefficients. 

%
Further, for $m=1$ we have $N_{1,k}=0$, for $k\le 4$, and $N_{1,k} = \lceil\frac k4\rceil$ for $k\ge 5$, while for $m>1$, $N_{m,k} \equiv N_k $ is independent of $m$.

\end{Conjecture}
\subsection*{Notes}
\ 
\begin{enumerate}
\item The case $m=0$ or $k=0$ of Conjecture~\ref{conj1}
is consistent with  the first Borwein conjecture, see \cite[Equation (1.1)]{An1995}. 
\item 
For given $m$ and $n$, 
the summation index $k$ is bounded by
$$k\le 4 n \binom{m+1}{2} = 2 m (m+1) n.$$
\item  For $m=1$, we must have $n\geq k/4$. Indeed, 
$n=\lceil\frac k4\rceil$
are the values of $N_{m,k}$ in Table~1 for $m=1$ for $k\ge 5$. For $k<5$, $\lceil\frac k4\rceil =1$,
so we have $N_{m,k} =0$, since for $n=0$ the statement of the conjecture holds trivially.
\item We examined the products for $m = 1, 2, \dots, 10$; $k = 0, 1, 2, \dots, 15$;
and $n=0, 1, 2,$ $\dots , 25$.  
For fixed $m$ and $k$, the value of $N_{m,k}$ such that the coefficient of $p^k$ in the products satisfy the Borwein $+ - - $ condition for $N_{m,k}\le n\le 25$ (for $m\le 5$) are recorded in Table~1. The  
values for $m=6, 7, \dots, 10$ were the same as for $m=5$. Thus for $m>1$, the values of $N_{m,k}$ appear to be independent of $m$. 
\item  The coefficients of $A_{m,n,k}(q)$ were non-negative for all the values of $m, n,$ and $k$ that we computed. 
\item The coefficients of powers of $q$ in $q^2C_{m,n,k}(q^3)$ are the same
as those of $qB_{m,n,k}(q^3)$, but in reverse order, that is, we have,
$$q^{n^2-1} B_{m,n,k} (q^{-1}) = C_{m,n,k}(q).$$
This can be seen by
replacing $q$ by $q^{-1}$ in \eqref{eq:conj1} and comparing the two sides.
\item One can ask, as did Stanton for \cite[Conjecture~3]{DS1999}, 
whether Conjecture~\ref{conj1} holds for
$n=\infty$.  However, this question is not applicable here, since the product on the
left-hand side of \eqref{eq:conj1} is not defined at $n=\infty$. 
 \end{enumerate}

\begin{center}
\begin{table}\label{table:conj1}
\begin{tabular}{|c | c| c | c | c | c | c | c | c |c | c | c | c | c | c | c | c | }
\hline
$m$ {\textbackslash} $k$ & 0 & 1 & 2 & 3 & 4  & 5 & 6 & 7 & 8  & 9 & 10 & 11 & 12  & 13 & 14 & 15   \cr
\hline
1 & 0 & 0 & 0 & 0  & 0 & 2 & 2 & 2  & 2 & 3 & 3 & 3  & 3 & 4 & 4 & 4 \cr
\hline
2 & 0 & 0 & 0 & 5  & 5 & 8 & 8 & 11  & 12 & 14 & 15 & 17  & 18 & 20 & 21 & 23 \cr
\hline
3 & 0 & 0 & 0 & 5  & 5 & 8 & 8 & 11  & 12 & 14 & 15 & 17  & 18 & 20 & 21 & 23 \cr
\hline
4 & 0 & 0 & 0 & 5  & 5 & 8 & 8 & 11  & 12 & 14 & 15 & 17  & 18 & 20 & 21 & 23 \cr
\hline
5 & 0 & 0 & 0 & 5  & 5 & 8 & 8 & 11  & 12 & 14 & 15 & 17  & 18 & 20 & 21 & 23 \cr
\hline
\end{tabular}
\vspace{5pt}

\caption{Apparent values of $N_{m,k}$, for $m=1, 2, \dots, 5$ and $k = 0, 1, \dots, 15$ }
\end{table}
\end{center}

We now make a few remarks about the form of Conjecture~\ref{conj1}. 
The modified theta function is defined as $$\theta(a; p) := (a; p)_\infty (p/a;p)_\infty .$$ 
Here we take $n=\infty$  and replace $q$ by $p$ in the definition of the $q$-shifted factorial. This product is convergent if $|p|<1$.
Consider the {theta shifted factorials} defined as \cite[eq. (11.2.5)]{GR90}
$$(a; q, p)_n := \prod_{i=0}^{n-1}\theta(a q^i;p) =
 \prod_{i=0}^{n-1} \prod_{j=0}^{\infty} \big(1-a p^j q^i\big)\big(1-p^{j+1}q^{-i}/a\big) .$$
As a natural extension of the Borwein Conjecture, consider 
$$(q; q^3,p)_n (q^2; q^3,p)_n, $$
or,
$$
\prod_{i=0}^{n-1} \prod_{j=0}^{\infty} 
\big(1-p^j q^{3i+1}\big) \big(1-p^j q^{3i+2}\big)
\big(1-p^{j+1}q^{-3i-1}\big) \big(1-p^{j+1}q^{-3i-2}\big) .$$
The product in Conjecture~\ref{conj1} should now be transparent. It is obtained by truncating the infinite products indexed by $j$. 
Indeed, one can try even more general ways to truncate the products.  
\begin{Conjecture}\label{conj2}
Let $m_1$, $m_2$, $n_1$, $n_2$, $n_3,$ and $k$ be non-negative integers.
Let the Laurent polynomials $A(q)=A_{m_1, m_2, n_1, n_2, n_3, k}(q)$, 
$B(q)=B_{m_1, m_2, n_1, n_2, n_3, k}(q)$ and
$C(q)=C_{m_1, m_2, n_1, n_2, n_3, k}(q)$ be defined by
\begin{align}
&(q, q^2;q^3)_{n_1} \prod_{j=1}^{m_1}(p^jq,p^jq^2;q^3)_{n_2}
\prod_{j=1}^{m_2} (p^jq^{-1},p^jq^{-2};q^{-3})_{n_3}
\notag\\
&=\sum_{k\geq 0} p^k\left[A(q^3)-q B(q^3)-q^2
C(q^3)\right].
\end{align}
For given $k$, if Let $m_1, m_2\geq 1$, and 
$n_1$, $n_2$ and $n_3$ are large enough, then the polynomials 
$A(q)$, $B(q)$, and
$C(q)$ have non-negative coefficients.  
\end{Conjecture} 
\subsection*{Notes}
\ 
\begin{enumerate} 
\item Borwein's second conjecture~\cite[Equation (1.3)]{An1995} states that $$(q, q^2; q^3)^2_n $$ satisfies the Borwein $+ - - $ condition. 
If we take
$m_1=1$, $m_2=0$, $n_2=n_1$, $p=1$, and ignore 
the condition $m_1,m_2\geq 1$, then the statement of Conjecture~\ref{conj2}, 
reduces to Borwein's second conjecture. 

\item  Andrews' refinement of Borwein's first two conjectures
\cite[Equation (1.5), $x=p$]{An1995} states that for each $k$, the coefficient of $p^k$ in
$$(q, q^2; q^3)_{n_1} (pq, pq^2; q^3)_{n_2} $$ satisfies the Borwein $+ - - $ condition.
Ae Ja Yee kindly informed us (private communication, January 2019), that Andrews' refinement does not hold. For example, it fails for $n_1=1$, $n_2=40$, and $k=40$. Again, if we take 
$m_1=1$ and $m_2=0$, the statement of Conjecture~\ref{conj2} reduces to Andrews' refinement of Borwien's first two conjectures. 

\item Our numerical experiments suggest that we must have $m_1, m_2\geq 1$ in Conjecture~\ref{conj2}. But the data we generated does not contradict Borwein's second conjecture. Further, it
may still be true that Andrews' refinement of Borwein's conjectures is true for large enough values of $n_1$ and $n_2$. 

\item It appears that Table 1 is relevant to Conjecture~\ref{conj2} too.  We observed the following from the data we generated. 
Let $k$ be fixed, and  $m_1, m_2\ge 2$. Let $n=\min\{n_1,n_2,n_3\}$.  Now if 
$n \ge N_k$, where $N_k\equiv N_{2,k}$ is taken from Table 1, the coefficients of $p^k$ in the expansion of the products in question satisfy the  Borwein $+ - - $ condition.  
\end{enumerate}

Next, on the suggestion of Dennis Stanton, we examine a conjecture due to Ismail, Kim and Stanton~\cite[Conjecture~1]{IKS1999} (see 
also Stanton~\cite[Conjecture~3]{DS1999}), who considered
$$ (q^a, q^{K-a};q^K)_n = \sum_{m=0}^\infty a_m q^m,$$
where  $a$ and $K$ are relatively prime integers with $a<K/2$. These authors
conjectured:\newline
If $K$ is odd, then
\begin{align*}
a_m\ge 0 & \text{ if } m \equiv \pm aj \mod K, \text{ for some non-negative even integer } j<K/2, \cr
\intertext{and,}
a_m\le 0 & \text{ if } m \equiv \pm aj \mod K, \text{ for some positive odd integer } j<K/2
.
\end{align*}

In \cite{DS1999}, this conjecture is followed by the statement:
If $K$ is even, then $(-1)^m a_m \geq 0$. The unfortunate placement of this statement suggests that it is part of the conjecture. 
 In fact, it is easy to prove. 
Since $a$ is relatively prime to $K$, and $K$ is even, both $a$ and $K-a$ are odd. Thus all the factors in the product are of the form $(1-q^{\text{odd}})$. Now to obtain a term $q^m$ with $m$ even,  we will need to multiply an even number of monomials of the form $(-q^{\text{odd}})$, so the sign will be positive. Similarly, if $m$ is odd, the sign will be negative. 
 
As in Conjecture~\ref{conj2}, we consider the formal expression
$$ (q^a;q^K,p)_n(q^{K-a};q^K,p)_n,$$
truncate the infinite products,
and check whether the coefficients satisfy a similar sign pattern. For $K$ even, it is easy to 
see that an analogous statement holds for the coefficient of $p^k$ for all non-negative integers $k$. 

For $K$ odd, 
we found that the sign pattern is the same as mentioned above, but only when $a=\lfloor K/2 
\rfloor$. In this case, the pattern is an elegant extension of Borwein's $+--$. When $K$ is of the form $4l+1$ or 
$4l+3$, the sign pattern is as follows:
 \begin{align*}
K=4l+1:\hspace{0.75cm} & 
\underbrace{++\cdots+}_{l+1} \underbrace{- - \cdots -}_{2l} \underbrace{+ + \cdots +}_l\cr
K=4l+3:\hspace{0.75cm} 
 & \underbrace{++\cdots+}_{l+1} \underbrace{- - \cdots -}_{2l+2} \underbrace{+ + \cdots +}_l
\end{align*}
For example, when $K=5$, then the pattern is $+ + -  -  + $, and 
when $K=7$, then the pattern is $+ + - - - -  + $. (As before, the $+$ sign represents  a non-negative, and the $-$ sign represents a non-positive coefficient.)

In what follows, we have replaced $K$ by $2K+1$; we consider only the odd powers of the base $q$. 

\newpage
\begin{Conjecture}\label{conj3}
Let $m_1$, $m_2$, $n_1$, $n_2$, $n_3,$ and $k$ be non-negative integers. Let $K$ be any positive number. 
Let the Laurent polynomials $A_k(q)=A_{m_1, m_2, n_1, n_2, n_3, k,K}(q)$
 be defined by
\begin{multline}\label{conj3-prod}
(q^K, q^{K+1};q^{2K+1})_{n_1} \prod_{j=1}^{m_1}(p^jq^K,p^jq^{K+1};q^{2K+1})_{n_2}
\cr
\times  \prod_{j=1}^{m_2} (p^jq^{-K},p^jq^{-K-1};q^{-2K-1})_{n_3}
=\sum_{k\geq 0} p^k A_k(q),
\end{multline}
where $A_k(q)$ is a Laurent polynomial of the form
$$A_k(q)=\sum_M a_{M,k}q^M.$$
Let $l= \lfloor \frac{2K+1}{4}\rfloor$. For given $k$, and $K$,  
if $m_1, m_2\geq1 $, and  $n_1$, $n_2$ and $n_3$ are large enough, then the 
coefficients $a_{M,k}$ satisfy the following sign pattern: 
$$
a_{M,k}
= \begin{cases}
 \geq  0, & \text{ if } M \equiv 0,  \pm i \mod 2K+1, \text{ for } i = 1, 2, \dots, l, \cr 
\leq  0, & \text{ otherwise.} 
 \end{cases}
 $$
\end{Conjecture} 

\subsection*{Notes}
\ 
\begin{enumerate}
\item If $m_1=0=m_2$, then the products on the left-hand side of \eqref{conj3-prod} are a special case of those considered in \cite[Conjecture~1]{IKS1999}.
\item When $K=1$, Conjecture~\ref{conj3} reduces to Conjecture~\ref{conj2}.
 \item We gathered data for the following values of the variables systematically.
 \begin{align*}
 m_1, m_2 & \in \{ 2, 3\}, \cr
n_1, n_2, n_3 & \in \{ 1, 2, \dots, 5\} ,\cr
k & \in \{1, 2, \dots, 10\}, \cr
K & \in\{2, 3, 4, \dots, 14 \}.
 \end{align*}
 In addition, we considered many random values, with
 \begin{align*}
 m_1, m_2, n_1, n_2, n_3 & \in \{0, 1, \dots, 10\}, \cr
k & \in \{0,1,\dots,  30\}, \cr
K & \in\{1, 2, 3, 4, \dots, 20 \}.
 \end{align*}
 In case we obtained a set of values that did not satisfy the required sign pattern, we performed further computations with larger values of $n_1$, $n_2$ or $n_3$. 
 \item  In our experiments, we found only a few values 
 where the predicted sign pattern does not hold, even for large values of $n_1$, $n_2$ and $n_3$. All of these were with either $m_1=0$ or $m_2=0$. For example, when $m_1=4, m_2=0, K=3, k= 18$. In particular the coefficient of $p^{18}q^{26}$ is predicted to be negative, but is in fact $1$, when $n_1$ and $n_2$ are large. This is the reason for the condition $m_1, m_2\geq 1$ in the statements of Conjectures~\ref{conj2} and \ref{conj3}.

%
\end{enumerate}


\section{Multiple series representations}
In this section we extend Andrews' explicit expressions for the polynomials
$A_n(q) $, $B_n(q)$ and $C_n(q)$ of \eqref{borwein1} appearing in the first Borwein conjecture.
Andrews~\cite[Eqs.~(3.4)--(3.6)]{An1995} showed that
\begin{subequations}
\begin{align}
A_n(q)&=\sum_{\la=-\infty}^\infty(-1)^\la q^{\la(9\la+1)/2}
\begin{bmatrix}2n\\n+3\la\end{bmatrix},\\
B_n(q)&=\sum_{\la=-\infty}^\infty(-1)^\la q^{\la(9\la-5)/2}
\begin{bmatrix}2n\\n+3\la-1\end{bmatrix},\\
C_n(q)&=\sum_{\la=-\infty}^\infty(-1)^\la q^{\la(9\la+7)/2}
\begin{bmatrix}2n\\n+3\la+1\end{bmatrix},
\end{align}
\end{subequations}
where
\begin{equation*}
\begin{bmatrix}m\\j\end{bmatrix}=\begin{cases}
0,&\text{if $j<0$ or $j>m$,}\\
\displaystyle \frac{(q;q)_m}{(q;q)_j(q;q)_{m-j}},\quad&
\text{otherwise,}
\end{cases}
\end{equation*}
denotes the $q$-binomial coefficient.
We use a result of Kaneko~\cite{Kan1996b} from the
theory of basic hypergeometric series with Macdonald polynomial argument
(see \cite{JK1996, Mac2013}) to give analogous expressions for the functions involved in Conjecture~\ref{conj1}. 

Let $F_{m,n}(p,q)$ denote the left-hand side of \eqref{eq:conj1}. We first dissect it as follows.
$$F_{m,n}(p,q) = F^0_{m,n}(p,q^3) - qF^1_{m,n}(p,q^3) -q^2F^2_{m,n}(p,q^3).$$
Thus, we have the definitions: 
\begin{align*}
F^0_{m,n}(p,q):=&{}\sum_{k=0}^{2m(m+1)n}p^kA_{m,n,k}(q),\\
F^1_{m,n}(p,q):=&{}\sum_{k=0}^{2m(m+1)n}p^kB_{m,n,k}(q),\\
F^2_{m,n}(p,q):=&{}\sum_{k=0}^{2m(m+1)n}p^kC_{m,n,k}(q).
\end{align*}
We extend Andrews' identities by writing each $F^l_{m,n}(p,q)$ (for $l=0, 1, 2$) as a 
$(2m+1)$-fold sum.  

In the following, $\la$ is an integer partition. That is, $\la$ is any sequence 
 $$\la=(\la_1,\la_2, \dots,\la_n, \dots)$$
 of non-negative integers such that
$\la_1\ge\la_2\ge\cdots\ge\la_n\ge \cdots $,
and contains only finitely many non-zero terms, called the parts of $\la$. 
We use the
symbol  $|\la|:=\la_1+\la_2+\cdots $ and say $\la$ is a partition of $|\la |$. 
In slight misuse of notation we shall also use $\lambda$
to denote finite non-increasing sequences of integers
which are not necessarily all non-negative.
For such sequences $\lambda$ the symbol $|\lambda|$ is
understood to denote the sum of the elements of $\lambda$,
as one would expect. 

\begin{Theorem}\label{mreps}
For $l=0,1,2$ we have
\begin{align*}
&F^l_{m,n}(p,q)=(-1)^{\binom{l+1}2}\,p^{m(m+1)n}q^{-mn^2}\notag\\
&\times\sum_{\substack{n\ge\la_1\ge\la_2\ge\cdots\ge\la_{2m+1}\ge-n\\[1pt]
|\la|\,\equiv\,-l\pmod{3}}}
\bigg(\prod_{1\le i<j\le 2m+1}
\frac{(1-p^{j-i}q^{\la_i-\la_j})(p^{j-i+1};q)_{\la_i-\la_j}}
{(1-p^{j-i})(p^{j-i-1}q;q)_{\la_i-\la_j}}\notag\\
&\qquad\qquad\qquad\qquad\qquad\quad\times
\prod_{i=1}^{2m+1}
\frac{(p^{i-1}q;q)_{2n}}
{(p^{i-1}q;q)_{n-\la_i}(p^{2m+1-i}q;q)_{n+\la_i}}
\notag\\
&\qquad\qquad\qquad\qquad\qquad\quad\times
(-1)^{|\la|}p^{\sum_{i=1}^{2m+1}(i-1-m)\la_i}
\notag\\
&\qquad\qquad\qquad\qquad\qquad\quad\times
q^{\binom{\la_1+1}2+\cdots+\binom{\la_{2m+1}+1}2-\frac{|\la|+l}3}
\bigg).
\end{align*}
\end{Theorem}
\begin{Remark}
From the expression in Theorem~\ref{mreps} it is not obvious
that the functions $F^l_{m,n}(p,q)$ are actually polynomials in $p$
of degree $2m(m+1)n$. 
\end{Remark}

Before proving the theorem, we outline some background information from the theory of basic hypergeometric series with Macdonald polynomial argument.  For the definition of the Macdonald polynomials $P_\la(x_1,\dots,x_n;q,t)$
together with their most essential properties, we refer to
Macdonald's book \cite{Mac1995}.

In particular, the $P_\la(x_1,\dots,x_n;q,t)$ are homogenous in
$x_1,\dots,x_n$ of degree $|\la|$; we have, after scaling each $x_i$ by $z$,
\begin{equation}\label{hom}
P_\la(zx_1,\dots,zx_n;q,t)=z^{|\la|}P_\la(x_1,\dots,x_n;q,t). 
\end{equation}
We also make use of the principal specialization formula \cite[p.\ 343, Ex.\ 5]{Mac1995}: Let 
\begin{equation}\label{pspec}
P_{\la}(1,t,\dots,t^{n-1};q,t)
=t^{n(\la)}\prod_{1\le i<j\le n}\frac{(t^{j-i+1};q)_{\la_i-\la_j}}
{(t^{j-i};q)_{\la_i-\la_j}},
\end{equation}
where $\la$ has at most $n$ parts, and $n(\la)=\sum_{i=1}^n(i-1)\la_i$. 

We require the following lemma.
\begin{Lemma}\label{corms}
Let $N$ be a non-negative integer. Then
\begin{align*}
&\prod_{i=1}^n(zt^{1-i},z^{-1}qt^{i-1};q)_N=\notag\\
&\sum_{N\ge\la_1\ge\la_2\ge\cdots\ge\la_n\ge-N}
\bigg(\prod_{1\le i<j\le n}
\frac{(1-q^{\la_i-\la_j}t^{j-i})(t^{j-i+1};q)_{\la_i-\la_j}}
{(1-t^{j-i})(qt^{j-i-1};q)_{\la_i-\la_j}}\notag\\
&\qquad\qquad\qquad\qquad\qquad\times\prod_{i=1}^n
\frac{(qt^{i-1};q)_{2N}}
{(qt^{i-1};q)_{N-\la_i}(qt^{n-i};q)_{N+\la_i}}\notag\\
&\qquad\qquad\qquad\qquad\qquad\times
q^{\binom{\la_1+1}2+\cdots+\binom{\la_n+1}2}
t^{\sum_{i=1}^n(i-1)\la_i}(-z^{-1})^{|\la|}
\bigg).
\end{align*}
\end{Lemma}
\begin{proof} 
We use a reformulation
of a result by Kaneko~\cite[Lemma~2]{Kan1996b}.
Let $N$ be a non-negative integer. Then
\begin{align*}
&\prod_{i=1}^n(-x_iq,-x_i^{-1};q)_N=\notag\\
&\sum_{N\ge\la_1\ge\la_2\ge\cdots\ge\la_n\ge-N}
\bigg(\prod_{1\le i<j\le n}\frac{(qt^{j-i};q)_{\la_i-\la_j}}
{(qt^{j-i-1};q)_{\la_i-\la_j}}\notag\\
&\qquad\qquad\qquad\qquad\qquad\times
\prod_{i=1}^n\frac{(qt^{i-1};q)_{2N}}
{(qt^{i-1};q)_{N-\la_i}(qt^{n-i};q)_{N+\la_i}}
\notag\\
&\qquad\qquad\qquad\qquad\qquad\times
q^{\binom{\la_1+1}2+\cdots+\binom{\la_n+1}2}
\notag\\
&\qquad\qquad\qquad\qquad\qquad\times
(x_1\cdots x_n)^{\la_n}P_{\la-\la_n}(x_1,\dots,x_n;q,t)\bigg),
\end{align*}
where $\la-\la_n$ stands for the partition $(\la_1-\la_n,\dots,\la_n-\la_n)$.

In Kaneko's identity, we take $x_i=-z^{-1}t^{i-1}$, for $1\le i\le n$,
make use of the homogeneity \eqref{hom} and the principal specialization
in \eqref{pspec}, to obtain the lemma.
\end{proof}
\begin{proof}[Proof of Theorem~\ref{mreps}]
We first observe that the product on the left-hand side of \eqref{eq:conj1}
can be written as
\begin{align*}
&\prod_{j=0}^m(p^jq,p^jq^2;q^3)_n\prod_{j=1}^m
(p^jq^{-1},p^jq^{-2};q^{-3})_n\notag\\
&=p^{m(m+1)n}q^{-3mn^2}
\prod_{i=1}^{2m+1}(p^{-m+i-1}q^2,p^{m-i+1}q;q^3)_n.
\end{align*}
Next, we apply the $(n,N,z,q,t)\mapsto(2m+1,n,p^mq,q^3,p)$
case of Lemma~\ref{corms} to arrive at
\begin{align*}
&\prod_{j=0}^m(p^jq,p^jq^2;q^3)_n\prod_{j=1}^m
(p^jq^{-1},p^jq^{-2};q^{-3})_n
=p^{m(m+1)n}q^{-3mn^2}\notag\\
&\times\sum_{n\ge\la_1\ge\la_2\ge\cdots\ge\la_{2m+1}\ge-n}
\bigg(\prod_{1\le i<j\le 2m+1}
\frac{(1-p^{j-i}q^{3\la_i-3\la_j})(p^{j-i+1};q^3)_{\la_i-\la_j}}
{(1-p^{j-i})(p^{j-i-1}q^3;q^3)_{\la_i-\la_j}}\notag\\
&\qquad\qquad\qquad\qquad\qquad\qquad\times
\prod_{i=1}^{2m+1}
\frac{(p^{i-1}q^3;q^3)_{2n}}
{(p^{i-1}q^3;q^3)_{n-\la_i}(p^{2m+1-i}q^3;q^3)_{n+\la_i}}
\notag\\&\qquad\qquad\qquad\qquad\qquad\qquad\times
(-1)^{|\la|}p^{\sum_{i=1}^{2m+1}(i-1-m)\la_i}
\notag\\&\qquad\qquad\qquad\qquad\qquad\qquad\times
q^{3\binom{\la_1+1}2+\cdots+3\binom{\la_{2m+1}+1}2-|\la|}\bigg).
\end{align*}

By picking the 
coefficients of $q^l$
with $l$ belonging to a residue class modulo $3$, we obtain the theorem.
\end{proof}

\begin{Remark} 
We can obtain a more general multiseries expression for the products
\begin{align*}
&\prod_{j=0}^m(p^jq^a,p^jq^{2K+1-a};q^{2K+1})_n\prod_{j=1}^m
(p^jq^{-a},p^jq^{a-1-2K};q^{-2K-1})_n
\end{align*}
by following a similar analysis as carried out in the proof of Theorem~\ref{mreps},
where 
we apply the $(n,N,z,q,t)\mapsto(2m+1,n,p^m q^a,q^{2K+1},p)$ case of Lemma~\ref{corms}.
The case $a=K$ gives the products on the left-hand side of \eqref{conj3-prod}, 
with $n=n_1=n_2=n_3$ and $m=m_1=m_2$. 
\end{Remark}

\section*{Acknowledgements}
We thank Dennis Stanton and the anonymous referee for helpful suggestions.
The computational results presented here have been achieved in part using the
Vienna Scientific Cluster (VSC).
The research of the first author was partially supported by the
Austrian Science Fund (FWF), grant~F50-N15, in the framework of the
Special Research Program ``Algorithmic and Enumerative Combinatorics". 
The research of the second author was partially supported by the
Austrian Science Fund (FWF), grant~P~3205-N35.
Open access funding is provided by the University of Vienna.
%

\end{document}